\documentclass[a4paper,oneside,12pt]{article}

\usepackage[utf8]{inputenc}
\usepackage[T5]{fontenc}
\usepackage{textcomp}
\usepackage[english]{babel}
\usepackage[autolanguage]{numprint}
\usepackage{hyperref}
\usepackage{graphicx}
\usepackage{verbatim}
\usepackage{enumerate}
\usepackage{fullpage}

\usepackage{amsmath,amssymb,amsthm,stmaryrd}

\usepackage[all]{xy}

\usepackage[nottoc,numbib]{tocbibind}

\renewcommand\epsilon\varepsilon
\renewcommand\phi\varphi
\renewcommand\geq\geqslant
\renewcommand\leq\leqslant
\renewcommand\ln\log

\newcommand\NN{\mathbb{N}}
\newcommand\ZZ{\mathbb{Z}}

\newcommand\coloneq{\mathbin{:=}}
\newcommand\ab\allowbreak

\numberwithin{equation}{subsection}

\theoremstyle{definition}
\newtheorem{Def}{Definition}[section]
\theoremstyle{plain}
\newtheorem{Pro}[Def]{Proposition}
\newtheorem{Lem}[Def]{Lemma}
\newtheorem{The}[Def]{Theorem}
\newtheorem*{The*}{Theorem}
\newtheorem{Cor}[Def]{Corollary}
\newtheorem*{Cor*}{Corollary}

\theoremstyle{remark}
\newtheorem{Exe}[Def]{Example}
\newtheorem{Rem}[Def]{Remark}

\title{Picard group of the forms of the affine line and of the additive group}
\author{Raphaël Achet \footnote{Université Grenoble Alpes,
Institut Fourier, CS 40700, 38058 Grenoble cedex 09. \newline
raphael.achet@univ-grenoble-alpes.fr}}
\date{}

\begin{document}

\maketitle

\abstract{We obtain an explicit upper bound on the torsion of the Picard group 
of the forms of $\mathbb{A}^1_k$ and their regular completions. We also obtain 
a sufficient condition for the Picard group 
of the forms of $\mathbb{A}^1_k$ to be non trivial and we give 
examples of non trivial forms of $\mathbb{A}^1_k$ with trivial Picard groups.}


{\bf MSC2010 Classification Codes}: 20G15; 20G07; 14R10; 14C22; 14K30.

\tableofcontents

\section*{Introduction and statement of the main results}

With the recent progress in the structure of linear algebraic groups over an 
imperfect field \cite{CGP} \cite{totaro}, it seems to be possible to study 
their Picard group if the Picard groups of unipotent algebraic groups are 
known well enough. As every unipotent smooth connected algebraic group 
is an iterated extension of forms of $\mathbb{G}_{a,k}$
\cite[XVII 4.1.1]{SGAIII_2}, this motivates the study of the Picard group of 
forms of $\mathbb{G}_{a,k}$. 

In this article, we consider more generally forms of the affine line, since 
our geometric approach applies to this setting without additional 
difficulties.

Let $k$ be a field, let $X$, $Y$ be schemes (resp. group schemes) over $k$. 
Recall that $X$ is a form of $Y$ if there is a field $K\supset k$ such that 
the scheme (resp. group scheme) $X_K$ is isomorphic to $Y_K$. We also recall 
that the affine line $\mathbb{A}^1_k$ is the $k$-scheme ${\rm Spec}(k[t])$; 
and the additive group $\mathbb{G}_{a,k}$ is the algebraic group of underlying 
scheme $\mathbb{A}^1_k={\rm Spec}(k[t])$, which represent the group functor:
\[
\begin{array}{rcl}
 Sch/k^\circ &\rightarrow &Groups \\
 T& \mapsto & \left(\mathcal{O}(T),+\right).
\end{array}
\]

If $k$ is a perfect field, all forms of $\mathbb{A}^1_k$ and
$\mathbb{G}_{a,k}$ are trivial. 
But non trivial forms of $\mathbb{A}^1_k$ and $\mathbb{G}_{a,k}$ exist over 
every imperfect field $k$; there structure has been studied by P. Russell 
\cite{Rus}, G. Greither \cite{haut1}, T.~Kambayashi, M.~Miyanishi, and 
M.~Takeuchi \cite{KMT} and \cite{FALF}. In \cite{haut1} and \cite{KMT} the 
Picard group of some special forms of $\mathbb{A}^1_k$ and $\mathbb{G}_{a,k}$ 
is described \cite[Lem.~6.12.2]{KMT} and \cite[Lem.~5.6]{haut1}. In 
\cite{FALF} T.~Kambayashi and M.~Miyanishi have continued the study of the 
forms of the affine line, they have proved numerous results on the forms of 
the affine line and on their Picard group \cite[Th. 4.2]{FALF},
\cite[Pro. 4.3.2]{FALF} and \cite[Cor. 4.6.1]{FALF}.

More recently B. Totaro has obtained an explicit description of
the class of extensions of a smooth connected unipotent group $U$ by the
multiplicative group as a subgroup of
${\rm Pic}(U)$ \cite[Lem.~9.2]{totaro}.
He has then applied this description to the structure of commutative
pseudo-reductive groups \cite[Lem.~9.4]{totaro} and \cite[Cor.~9.5]{totaro}. 
Moreover he has constructed an example of a non trivial form of
$\mathbb{G}_{a,k}^2$, such that ${\rm Pic}(U_{k_s})$ is trivial 
\cite[Exa.~9.7]{totaro}.

In this article, we go back over and improve some of the results of 
\cite{FALF} and \cite{haut1} with different methods.

Given a form $X$ of $\mathbb{A}^1_k$, it is known that there exists a finite 
purely inseparable extension $K$ of $k$ such that $X_K \cong \mathbb{A}^1_K$; 
then ${\rm Pic}(X)$ is $p^m$-torsion, where ${p^m\coloneq [K:k]}$ (see e.g. 
\cite[Lem.~2.4]{MB_lin_line_b}). Our main theorem yields a sharper result:
\begin{The*}[\ref{main_th}]
Let $X$ be a non trivial form of $\mathbb{A}^1_k$, and let $n(X)$ be the 
smallest non-negative integer such that
$X_{k^{p^{-n}}}\cong \mathbb{A}^1_{k^{p^{-n}}}$.
\begin{enumerate}[(i)]
\item ${\rm Pic}(X)$ is $p^{n(X)}$-torsion.

\item If $X$  has a $k$-rational point (e.g. $X$ is a form of
$\mathbb{G}_{a,k}$ or $k$ is separably closed), then ${\rm Pic}(X)\neq \{0\}$.
\end{enumerate}
\end{The*}

Assertion $(i)$ is stated by T. Kambayashi and M. Miyanishi in
\cite[Pro.~4.2.2]{FALF}, but their proof is only valid when $k$ is separably
closed. The arguments of our proof of assertion $(i)$ are quite general: we 
use them to obtain a bound on the torsion of the Picard groups of some higher 
dimensional $k$-varieties (Proposition~\ref{main_th_gen}).
T.~Kambayashi and M. Miyanishi have also shown that the exponent of the Picard 
group of a form of the affine line having a $k$-rational point is at least
$p^{n(X)}$ (see \cite[Pro. 4.2.3]{FALF}); this implies assertion $(ii)$.
We provide a short alternative proof of that assertion.

A form of $\mathbb{A}^1_k$ does not necessary have a $k$-rational point. In 
Subsection \ref{pic_triv} we present an explicit example of such a form with 
trivial Picard group (Lemma \ref{lem_pic_trivial}), and a more general 
construction (Proposition \ref{pro_triv}).
We will also show that the non trivial forms of $\mathbb{G}_{a,k}$ are 
not special algebraic groups.
This result has already been shown by D. T. Nguyễn \cite{note_ngu}, we are 
going to use a different method: we see it as a corollary of
our main Theorem \ref{main_th} and Proposition~\ref{pro_triv}.

Next, we consider the regular completion $C$ of the curve $X$.
The Picard groups of $C$ and $X$ are linked by a standard exact sequence 
\eqref{exacte.1}. We obtain the following result on the Picard functor
${\rm Pic}^0_{C/k}$:

\begin{The*}[\ref{th_pic_X}]
Let $X$ be a form of $\mathbb{A}^1_k$ and $C$ be the regular completion of
$X$. Let $n'(X)$ be the smallest non-negative integer $n$ such that the 
function field of $X_{k^{p^{-n}}}$ is isomorphic to $k^{p^{-n}}(t)$.
Let $k'$ be the unique minimal field extension of $k$ such that
$X_{k'} \cong \mathbb{A}^1_{k'}$.

Then ${\rm Pic}^0_{C/k}$ is a smooth connected unipotent algebraic group of 
$p^{n'(X)}$-torsion which is $k$-wound and splits over $k'$.

In addition if $X$ is a non trivial form of $\mathbb{G}_{a,k}$ and $p\neq 2$, 
then $k'$ is the minimal field extension of $k$ such that ${\rm Pic}^0_{C/k}$ 
splits over $k'$.
\end{The*}

The full statement of Theorem \ref{th_pic_X} also contain an upper bound on 
the dimension of ${\rm Pic}^0_{C/k}$ for a class of forms of $\mathbb{A}^1_k$, but 
its 
formulation requires additional notations. This upper bound is obtain by 
computing the arithmetic genus of some curve in some weighted projective plane
(Corollary \ref{cor_deg_genre}).

The fact that ${\rm Pic}^0_{C/k}$ is smooth and $k$-wound is a direct 
consequence of results obtained in \cite[Chap.8 and Chap.9]{NM}. The fact that 
${\rm Pic}^0_{C/k}$ is unipotent is obtained in \cite[Th.~6.6.10]{KMT}.
We have the inequality $n(X) \geq n'(X)$, so Theorem~\ref{th_pic_X} yields a 
better bound on the torsion of ${\rm Pic}^0(C)$ than Theorem \ref{main_th} 
(see the exact sequence~\eqref{exacte.G.pic}). T. Kambayashi and M. Miyanishi 
obtained that the exponent of ${\rm Pic}^0_{C/k}$ is $p^{n'(X)}$ 
\cite[Cor.~4.6.1]{FALF}; this implies our result on the torsion of
${\rm Pic}^0_{C/k}$. We will provide an alternative proof of this result.

Before the proof of Theorem \ref{th_pic_X}, we will first gather in Section 
\ref{sec2} some results about the Picard functor, which are of independent 
interest. These results will be used in Section \ref{pic_fonct} to prove
Theorem \ref{th_pic_X}.
\medskip

{\bf Conventions}: Let $k$ be a field, unless explicitly stated, $k$ is of 
characteristic $p>0$. We choose an algebraic
closure $\overline{k}$, and denote by $k_s\subset \overline{k}$ the separable 
closure. For any non-negative integer $n$ we denote
$k^{p^{-n}}=\{x \in \overline{k} | x^{p^n} \in k\}$.

Let $X$ be a scheme; we note $\mathcal{O}_X$ the structural sheaf of 
$X$. We will denote $\mathcal{O}(X)$ the ring of regular functions on $X$, and
$\mathcal{O}(X)^*$ the multiplicative group of invertible regular functions 
on $X$. Let $x \in X$, the stalk of $\mathcal{O}_X$ at $x$ is denoted
$\mathcal{O}_{X,x}$, the residue field at $x$ is denoted $\kappa(x)$. 

The morphisms considered between two $k$-schemes are morphisms over $k$.
An {\it algebraic variety} is a scheme of finite type on ${\rm Spec}(k)$. Let 
$K$ be a field extension of $k$, the base change
$X \times_{{\rm Spec}(k)} {\rm Spec}(K)$ is denoted $X_K$. Let $X$ be an 
integral variety, the function field of $X$ is denoted $\kappa(X)$. A group 
scheme of finite type over $k$ will be called an {\it algebraic 
group}. A group scheme locally of finite type over $k$ will be called a
{\it locally algebraic group}.

A smooth connected unipotent algebraic group $U$ over $k$ is said to be
$k$-{\it split} if $U$ has a central composition series with successive 
quotients forms of $\mathbb{G}_{a,k}$. A smooth connected  unipotent algebraic 
group $U$ over $k$ is said to be $k$-{\it wound} if every $k$-morphism
$\mathbb{A}^1_k \rightarrow U$ is constant with image a point of $U(k)$ (an
equivalent definition is: $U$ does not have a central subgroup isomorphic to
$\mathbb{G}_{a,k}$ \cite[Pro. B.3.2]{CGP}).

\section[Forms of the affine line and of the additive group]{Forms of $\mathbb{A}^1_k$ and of $\mathbb{G}_{a,k}$}

\subsection{Regular completion and invariants}
\label{sec_intro}

In this first Subsection we introduce some notations and gather some results 
from P. Russell's article~\cite{Rus} that we will use in the rest of the 
article.

Let $X$ be a form of $\mathbb{A}^1_k$, we will note $C$ its regular 
completion, i.e., the unique projective regular curve such that there is an 
open dominant immersion ${j:X \rightarrow C}$ satisfying the following 
universal property: for every morphism $f:X \rightarrow Y$ to a proper scheme 
$Y$ there exists a unique morphism $\hat{f}: C \rightarrow Y$ such that
${\hat{f} \circ j = f}$ \cite[Th. 15.21]{AG.Gortz_Torsen}.

\begin{Lem}\label{rus1.1}\cite[1.1]{Rus}

Let $X$ be a form of $\mathbb{A}^1_k$, let $C$ be the regular completion of
$X$.
\begin{enumerate}[(i)]
\item $C\setminus X$ is a point denoted $P_\infty$ which is purely inseparable 
over $k$.

\item There is a unique minimal field extension $k'$ such that
$X_{k'} \cong \mathbb{A}^1_k$, and $k'$ is purely inseparable of finite 
degree over $k$.
\end{enumerate}
\end{Lem}

Let $\phi_k$ be the Frobenius morphism of $k$, i.e. the morphism 
\[\phi_k:x\in k \mapsto x^{p} \in k.\]
In the following we will denote $\phi$ for $\phi_k$.

Let $X$ be a form of $\mathbb{A}^1_k$, by definition $X={\rm Spec}(R)$ with $R$ a 
$k$-algebra such that $R\otimes_k k' \cong k'[t]$. Let $n$ be a non-negative integer,
we consider
\[\begin{array}{rrcl}
F_R^{n}:&R \otimes_k k &\rightarrow  &R \\
& r\otimes x & \mapsto & xr^{p^n} 
\end{array}\]
with $k$ seen as a $k$-algebra via $\phi^n$, the $n$th power of $\phi$.

The morphism $F_R^{n}$ corresponds at the scheme level to the $n$th relative 
Frobenus morphism $F_X^n$. Let  $X^{(p^n)}$ be the base change
$X \times_{{\rm Spec}(k)} {\rm Spec}(k)$ with $k$ seen as a $k$-algebra via
$\phi^n$, in other world $X^{(p^n)}$ is isomorphic to the base change of $X$ 
by $k^{p^{-n}}$. We can then write
\[F^n_X: X \rightarrow X^{(p^n)}.\]

\begin{Lem}\cite[1.3]{Rus}

There is an integer $n\geq 0$ such that
$X^{(p^n)} \cong \mathbb{A}^1_{k^{p^{-n}}}$.
\end{Lem}

\begin{Def}\label{def_n'}
Let $X$ be a form of $\mathbb{A}^1_k$.
\begin{enumerate}[(i)]
\item The smallest non-negative integer $n$ such that
$X^{(p^n)} \cong \mathbb{A}^1_{k^{p^{-n}}}$ is denoted $n(X)$.

\item The smallest non-negative integer $n$ such that
$\kappa\left(X^{(p^n)}\right)\cong k^{p^{-n}}(t)$  is denoted $n'(X)$.

\item The point $P_\infty$ is purely inseparable (Lemma \ref{rus1.1}), let $r(X)$ be the 
integer such that ${\rm deg}(P_\infty)=p^{r(X)}$.
\end{enumerate}
\end{Def}

\begin{Rem}
\begin{enumerate}[(i)]
\item We have $n(X) \geq n'(X)$, we will show in Example \ref{C_P1} that this
inequality can be strict (but equality holds if $X$ is a form of
$\mathbb{G}_{a,k}$ and if $p\neq 2$ according to Lemma \ref{X_P1=>p=2}).

\item Let $n$ be $n(X)$, the morphism $F^n_X$ extend to a finite dominant morphism
$\mathcal{F}^n_X:C \rightarrow \mathbb{P}^1_{k^{p^{-n}}}$ of degree $p^n$ 
\cite[Lem 1.3]{Rus}.
Then $p^{r(X)}$ is the residue class degree of the valuation associated to $P_\infty$ in
$\kappa(C)$, so
\[
p^{r(X)}=[\kappa(P_\infty):k]\leq \left[\kappa(C):
\kappa\left(\mathbb{P}^1_{k^{p^{-n}}}\right)\right]=p^{n}.
\]
Hence $r(X)\leq n(X)$.
\end{enumerate}
\end{Rem}

\begin{Def} \label{def_m}
Let $m(X)$ be the positive integer such that the image of the group morphism
${\rm deg}: {\rm Pic}(C) \rightarrow \ZZ$ is $m(X)\ZZ$.
\end{Def}

\begin{Rem}
$m(X)$ is the greatest common divisor of the degrees of the residue fields of the closed 
points of $C$, in particular $m(X)$ divides $[\kappa(P_\infty):k]=p^{r(X)}$. So $m(X)$ is 
a power of $p$ and $m(X) \leq p^{r(X)}$. 
\end{Rem}

We have shown the following relations between the above invariants:
\begin{Lem} \label{lem_ineq}
\begin{align*}
n(X)&\geq \max\left(n'(X),r(X)\right)\\ m(X)& \ | \ p^{r(X)}.
\end{align*}
\end{Lem}

\subsection[Structure of the forms of the additive group]{Structures of the forms of $\mathbb{G}_{a,k}$}

In this Subsection we will gather some results mainly from P. Russell's 
article \cite{Rus} on the structure of the forms of $\mathbb{G}_{a,k}$ and on 
the reasons why a form of $\mathbb{A}^1_k$ can fail to have a group structure.

Let $A={\rm End}_k(\mathbb{G}_{a,k})$ (endomorphisms of $k$-group scheme) and
$F=F^{1}_ {\mathbb{G}_{a,k}} \in A$ 
the relative Frobenius endomorphism. Then $A=k\langle F \rangle$ is a non commutative 
ring of polynomials with the relations $Fa=a^pF$ for all $a\in k$. Following \cite{Rus}, 
we denote by $A^*$ the subset of polynomials in $A$ with non zero constant coefficients.

\begin{The}\label{rus.2.1} \cite[2.1]{Rus}

Let $G$ be a form of $\mathbb{G}_{a,k}$.
Then $G$ is isomorphic to the subgroup
${\rm Spec} \left( k[x,y]/I \right)$ of $\mathbb{G}_{a,k}^2$, where $I$ is the ideal
of $k[x,y]$ generated by the separable polynomial 
${y^{p^n}-\left(x+a_1x^p+\dots+a_mx^{p^m}\right)}$ for some $a_1, \dots, a_m \in k$.
Equivalently, $G$ is the kernel of the homomorphism
\begin{equation}\label{eq_Ga}
\begin{array}{rcl}
\mathbb{G}_{a,k}^2 &\rightarrow & \mathbb{G}_{a,k} \\
(x,y) & \mapsto & y^{p^n}-\left(x+a_1x^p+\dots+a_mx^{p^m} \right).
\end{array}
\end{equation}
Thus, we can see $G$ as a fibre product
\[\xymatrix{
G \ar[r] \ar[d]  & \mathbb{G}_{a,k} \ar[d]^{\tau} \\
\mathbb{G}_{a,k} \ar[r]_{F^n} & \mathbb{G}_{a,k}
}\]
where $\tau=1+a_1F+\dots+a_mF^m \in A^*$. Similarly, any $G$ defined by such a product
is a form of $\mathbb{G}_{a,k}$. We note $G=\left(F^n,\tau \right)$.
\end{The}

\begin{Rem}\label{n(G)}
Let $G$ be a form of $\mathbb{G}_{a,k}$, the proof of \cite[Th.~2.1]{Rus} 
shows that in the equation \eqref{eq_Ga} we can choose $n$ to be $n(G)$.
\end{Rem}

Recall that any smooth connected unipotent algebraic group splits after base 
change by a finite purely inseparable extension \cite[Cor.~IV § 2 3.9]{DG}. In 
the particular case of the forms of $\mathbb{G}_{a,k}$, we have the following 
more precise result:

\begin{Cor}\cite[2.3.1]{Rus}

Let $G$ be the form of $\mathbb{G}_{a,k}$ defined by the equation 
${y^{p^n}=x+a_1x^p+\dots+a_mx^{p^m}}$. Then
${k'\coloneq k\left(a_1^{p^{-n}},\dots,a_m^{p^{-n}}\right)}$  if the smallest 
extension of $k$ such that $G_{k'} \cong \mathbb{G}_{a,k'}$.
\end{Cor}

Let $X$ be a form of $\mathbb{A}^1_k$, P. Russell showed in his article 
\cite{Rus} that there are two reasons for $X$ to fail to have a group 
structure. Firstly $X$ may not have a $k$-rational point. Secondly $X_{k_s}$ 
may have only finitely many automorphisms.

\begin{Pro}\cite[6.9.1]{KMT}

Let $X$ be a form of $\mathbb{A}^1_k$, such that $X$ has a $k$-rational point 
$P_0$. Let $C$ be the regular completion of $X$. Then the following are 
equivalent:
\begin{enumerate}[(i)]
\item $X$ has a group structure with neutral point $P_0$.

\item $X$ is isomorphic as a scheme to a form of $\mathbb{G}_{a,k}$.

\item ${\rm Aut}(C_{k_s})$ is infinite.
\end{enumerate}
\end{Pro}

\begin{Pro} \label{homo} \cite[4.1]{Rus}

Let $X$ be a form of $\mathbb{A}^1_k$ and suppose that $X_{k_s}$ admits a 
group structure. Then $X$ is a principal homogeneous space for a form $G$ of
$\mathbb{G}_{a,k}$ determined uniquely by $X$. Moreover
$X={\rm Spec}(k[x,y]/I)$, $G={\rm Spec}(k[x,y]/J)$ where the ideals $I$ and
$J$ are generated respectively by $y^{p^n}-b-f(x)$ and
$y^{p^n}-f(x)$ with $b\in k$ and $f(x)\coloneq x+a_1x^p+\dots+a_mx^{p^m}$. 
Conversely, if $X$ and $G$ are defined as above, then $X$ is a principal 
homogeneous space for $G$.
\end{Pro}

\begin{Rem}
P. Russell in \cite{Rus} and T.~Kambayashi, M.~Miyanishi, and M.~Takeuchi in 
\cite{KMT} have classified all forms of $\mathbb{A}^1_k$ over a separably 
closed field such that the regular completion has arithmetic genus $\leq 1$.

M. Rosenlicht \cite[6.9.3]{KMT} has found an example of a form of
$\mathbb{A}^1_k$ with only finitely many automorphisms, of genus $(p-1)/2$ for 
all $p>2$.

More recently, T. Asanuma \cite[Th. 8.1]{PI_forms_aff_curves} has found an
explicit algebraic presentation of the forms of $\mathbb{A}^1_k$, for every 
field $k$ of characteristic $p>2$.
\end{Rem}

\subsection{Examples}

Let $X$ be a form of $\mathbb{A}^1_k$, first we will compare the minimal field 
$k'$ such that $X_{k'}\cong \mathbb{A}^1_{k'}$ and the residue field
$\kappa(P_\infty)$ of $P_\infty$. There is an inclusion
${\kappa(P_\infty)\subset k'}$, which may be strict, as shown by the example 
below.

\begin{Exe}
Let $k=\mathbb{F}_p(t_1,t_2)$ and $G$ be the form of $\mathbb{G}_{a,k}$ 
defined by the equation
\[y^{p^2}=x+t_1x^p+t_2x^{p^2},\]
then $C$ is defined as a curve of $\mathbb{P}^2_k$ by the equation
\[y^{p^2}=xz^{p^2-1}+t_1x^pz^{p^2-p}+t_2x^{p^2}.\]
In this case
$\kappa(P_\infty)=k(t_2^{p^{-2}}) \varsubsetneq k'=k(t_1^{p^{-2}},t_2^{p^{-2}})$.
\end{Exe}

The inequality $p^{r(X)}=[\kappa(P_\infty):k] \leq p^{n(X)}$ (Lemma \ref{lem_ineq}) may
also be strict, as shown by:

\begin{Exe}\label{exe_batard}
Let $k=\mathbb{F}_p(t)$ and $G$ be the form of $\mathbb{G}_{a,k}$ 
defined by the equation 
\[y^{p^3}=x+tx^p+t^{p^2}x^{p^2},\]
then $n(G)=3$ and after the change of variable $w=tx-y^p$ we remark that $G$ is also 
defined by the equation
\[-t^{1-p}y^{p^2}-t^{-1}y^p=t^{-1}w+t^{1-p}w^p+w^{p^2}.\]
So $C$ is defined in $\mathbb{P}^2_k$ by
\[-t^{1-p}y^{p^2}-t^{-1}y^pz^{p^2-p}=t^{-1}wz^{p^2-1}+t^{1-p}w^pz^{p^2-p}+w^{p^2},\] 
the residue field of the point at infinity is $\kappa(P_\infty)=k(t^{p^{-2}})$.
\end{Exe}

We will now present some results on the forms of $\mathbb{A}^1_k$ with regular 
completion equal to $\mathbb{P}^1_k$.

\begin{Lem} \label{lem_X_P1}
Let $X$ be a form of $\mathbb{A}^1_k$ such that $X(k)\neq \emptyset$. The following are 
equivalent:
\begin{enumerate}[(i)]
\item $ C \cong \mathbb{P}_k^1$.

\item $X$ is the complement of a purely inseparable point of $\mathbb{P}_k^1$.

\item $C$ is  smooth.
\end{enumerate}
\end{Lem}

\begin{proof}
We begin with $(i) \Leftrightarrow (ii)$, the implication $(i) \Rightarrow (ii)$ is
a consequence of \cite[Lem.~1.1]{Rus}. The converse is clear.

Now we show $(i) \Leftrightarrow (iii)$, the implication $(i) \Rightarrow (iii)$ is 
clear. Suppose $C$ is smooth, let $k'$ the smallest field such that
$X_{k'} \cong \mathbb{A}^1_{k'}$. Then $C_{k'}$ is smooth; so
$C_{k'}\cong \mathbb{P}^1_{k'}$ and $C(k) \neq \emptyset$. According to
\cite[Pro.~7.4.1 (b)]{QL} it follows that $C \cong \mathbb{P}_k^1$.
\end{proof}

\begin{Rem}\label{rem_rat}
If $X$ is a non trivial form of $\mathbb{A}^1_k$, then $P_\infty$ is not $k$-rational. 
Indeed if $P_\infty$ is $k$-rational then $C$ is smooth at $P_\infty$
\cite[Pro.~4.3.30]{QL} so it is smooth everywhere. According to Lemma \ref{lem_X_P1} $C$ 
is isomorphic to $\mathbb{P}^1_k$ and $X$ is the complement of a $k$-rational point of
$\mathbb{P}^1_k$, thus $X\cong \mathbb{A}^1_k$.
\end{Rem}

\begin{Lem}\label{X_P1=>p=2} \cite{Rosen_aut} \cite{Rus} \cite[6.9.2]{KMT}

Let $G$ be a form of $\mathbb{G}_{a,k}$.
If $C \cong \mathbb{P}^1_k$ then either $G \cong \mathbb{G}_{a,k}$ or $p=2$ and $n(G)=1$.
\end{Lem}

\begin{Exe}\label{C_P1}
Let $p=2$, and $G$ be the form of $\mathbb{G}_{a,k}$ defined by the equation
\[y^2=x+ax^2\]
with $a\in k \setminus k^2$ where $k^2=\{x^2 \ | \ x\in k \}$. Then $G$ 
is a non trivial form of $\mathbb{G}_{a,k}$, the regular completion $C$ is defined as a 
curve of $\mathbb{P}^2_k$ by the equation
\[y^2=xz+ax^2.\]
We remark that $C$ is smooth (because it is smooth at $P_\infty$), so according to 
Lemma~\ref{lem_X_P1} $C\cong \mathbb{P}^1_k$ (this follows more directly from the fact 
that $C$ is a conic with a $k$-rational point).
\end{Exe}

\begin{Rem}
We can combine examples \ref{exe_batard} and \ref{C_P1}: let $p=2$ and $G$ be the form of
$\mathbb{G}_{a,k}$ defined by
\[
y^{p^3}=x+tx^p+t^{p^2}x^{p^2},
\]
then $r(G)=2$ and $n(G)=3$. Moreover $G_{k^{p^{-2}}}$ is isomorphic to the form of
$\mathbb{G}_{a,k}$ defined by the equation $y^2=x+t^{p^{-2}}x^2$, so $n'(G)=2$. So we
have constructed an example of a form of $\mathbb{G}_{a,k}$ where the inequality
$n(X) \geq \max(n'(X),r(X))$ (Lemma~\ref{lem_ineq}) is strict.
\end{Rem}

\begin{Exe}
Let $Q$ be an inseparable point of $\mathbb{P}^1_k$, then
$X=\mathbb{P}_k^1\setminus \{Q\}$
is a form of $\mathbb{A}^1_k$ with regular completion $\mathbb{P}^1_k$. If $Q$ 
is not $k$-rational then $X$ is a non trivial form of $\mathbb{A}^1_k$ and if 
${\rm deg}(Q)>2$ then according to Lemma \ref{X_P1=>p=2}, $X_{k_s}$ does not 
have a group structure. In this case $n'(X)=0$ and $n(X)=r(X)$ can be 
arbitrary big.
\end{Exe}

\begin{Exe}
Let $X$ be a form of $\mathbb{A}^1_k$, if $C \cong \mathbb{P}_k^1$ then
$k'=\kappa(P_\infty)$. The converse is false: let $G$ be the form of $\mathbb{G}_{a,k}$ 
defined by the equation $y^p=x+ax^p$ where $a \in k\setminus k^p$. Then $C$ is defined by 
the equation $y^p=xz^{p-1}+ax^p$ in $\mathbb{P}^2_k$, so
$\kappa(P_\infty)=k[a^{p^{-1}}]=k'$. If $p\geq3$, then $C$ isn't smooth (because $C_{k'}$ 
is not regular at $P_\infty$) so $C$ is not isomorphic to $\mathbb{P}^1_k$.
\end{Exe}

\subsection{Arithmetic genus of the regular completion}
\label{genre_naif}

First let us consider a field $k$ of arbitrary characteristic.
Let $a$, $b$ and $c$ be three positive integers, recall that the weighted projective space
$\mathbb{P}_k(a,b,c)$ is defined as ${\rm Proj}(k[x,y,z])$ where $k[x,y,z]$ is the graded 
polynomial $k$-algebra with weight $a$ for $x$, $b$ for $y$ and $c$ for $z$.
If $w$ is an homogeneous element of $k[x,y,z]$, we will denote
$D_+(w)$ the open subset of $\mathbb{P}_k(a,b,c)$ consisting of the homogeneous 
ideals of ${\rm Proj}(k[x,y,z])$ not containing the ideal $(w)$. Then
$\left(D_+(w),{\mathcal{O}_{\mathbb{P}_k(a,b,c)}}_{|D_+(w)}\right)$ is an affine scheme.

Let $C$ be a geometrically integral curve of degree $d$ in $\mathbb{P}_k^2$, we
denote  $p_a(C)$ the arithmetic genus of the curve $C$. It is well known that 
$p_a(C)={(d-1)(d-2)/2}$. In this Subsection we will generalize this result for some curves 
in some weighted projective planes (Proposition \ref{pro_genre_deg}). 
I. Dolgachev has computed the geometric genus of a smooth curve in a 
weighted projective plane \cite[3.5.2]{WPS} under the assumption that the 
characteristic of the field does not divide the weights of the projective 
plane. But we need to compute 
the arithmetic genus of a curve in a weighted projective plane where one of 
the weights is a power of the characteristic (Corollary \ref{cor_deg_genre}). 
Even though the result of 
Proposition \ref{pro_genre_deg} is certainly already known, we did not find a 
reference with an appropriate setting so we include the proof here for the sake of completeness.

\begin{Lem}\label{lem_genre}
Let $k$ be a field of arbitrary characteristic, let $a$ be a positive integer and let $n$ 
be an integer. Let $S=\bigoplus\limits_{d \in \NN} S_d$ be the graded polynomial
$k$-algebra $k[x,y,z]$ with weight $1$ for $x$, $1$ for $y$ and $a$ for $z$.
We denote $\mathbb{P}$ the weighted projective space $\mathbb{P}_k(1,1,a)={\rm Proj}(S)$.

Then $\mathcal{O}_\mathbb{P}(na)$ is an invertible sheaf on $\mathbb{P}$ and
${\rm H}^0\left(\mathbb{P},\mathcal{O}_\mathbb{P}(na)\right)=S_{na}$.
\end{Lem}

\begin{proof}
First we will show that
$S_{n a}={\rm H}^0\left(\mathbb{P},\mathcal{O}_\mathbb{P}(n a)\right)$ (in the 
case where $a=1$ and $\mathbb{P}=\mathbb{P}^2_k$ it is a well known fact).
Let $g \in \mathcal{O}_\mathbb{P}(n a)(\mathbb{P})$, then by definition of
$\mathcal{O}_\mathbb{P}(n a)$:
\begin{align*}
g_{|D_+(x)}&=P/x^{m_x} \ {\rm with} \ P \in S_{m_x+n a} {\rm ,} \\
g_{|D_+(y)}&=Q/y^{m_y} \ {\rm with} \  Q \in S_{m_y+n a}{\rm .}
\end{align*}
We can suppose that $m_x=m_y=m$ (if this not the case, for example $m_x>m_y$, consider
$Q'=Qy^{m_x-m_y}$, then $g_{|D_+(y)}=Q'/y^{m_x}$).
The two local sections $g_{|D_+(x)}$ and $g_{|D_+(y)}$ coincide on $D_+(xy)$, so
\[
g_{|D_+(xy)}=P/x^{m}=Q/y^{m}\in S\left[\frac{1}{xy}\right].
\]
Then, in $S$ we have the equality $x^{m}Q=y^{m}P$.
So $y^{m}$ divide $Q$, and $g=Q/y^{m}$ is a homogeneous polynomial
of degree $m+n a-m=n a$. Thus $g \in S_{n a}$, so
${\rm H}^0\left(\mathbb{P},\mathcal{O}_\mathbb{P}(n a)\right) \subset S_{n a}$.
Conversely, it is clear that $S_{n a}\subset
{\rm H}^0\left(\mathbb{P},\mathcal{O}_\mathbb{P}(n a)\right)$.

Next, to show that $\mathcal{O}_\mathbb{P}(na)$ is an invertible sheaf on
$\mathbb{P}$, it is enough to show that for $U=D_+(x)$, $D_+(y)$ and $D_+(z)$, the
$\mathcal{O}_\mathbb{P}(U)$-module $\mathcal{O}_\mathbb{P}(na)(U)$ is isomorphic to
$\mathcal{O}_\mathbb{P}(U)$.

Let $w$ be $x$ or $y$, the multiplication by $w^{na}$:
\[
\begin{array}{rrcl}
{\rm mult}_{w^{na}}:&\mathcal{O}_\mathbb{P}(D_+(w))&
\rightarrow &\mathcal{O}_\mathbb{P}(na)(D_+(w)) \\
&P/w^{{\rm deg}(P)}&\mapsto&w^{na}P/w^{{\rm deg}(P)},
\end{array}
\]
has for inverse the multiplication by $1/w^{na}$. So ${\rm mult}_{w^{na}}$ is an 
isomorphism.

For $D_+(z)$, the isomorphism is the multiplication by $z^n$:
\[
\begin{array}{rrcl}
{\rm mult}_{z^{n}}:&\mathcal{O}_\mathbb{P}(D_+(z))&
\rightarrow& \mathcal{O}_\mathbb{P}(na)(D_+(z)) \\
&P/z^{m}&\mapsto&z^{n}P/z^{m},
\end{array}
\]
where $P \in S_{ma}$.
\end{proof}

\begin{Pro}\label{pro_genre_deg}
Let $k$ be a field of arbitrary characteristic,
let $a$ be a positive integer. We denote $\mathbb{P}$ the weighted projective space
$\mathbb{P}_k(1,1,a)$.

Let $C$ be a geometrically integral curve of degree $d$ in $\mathbb{P}$, such that $d$ is 
a multiple of $a$. Let $h$ be the integer $d/a$. Then the arithmetic genus of $C$ is:
\[
p_a(C)=\frac{(h-1)(d-2)}{2}.
\]
\end{Pro}

\begin{proof}
Let $n$ be a positive integer, according to Riemann-Roch Theorem \cite[Th.~7.3.17]{QL},
\begin{equation}\label{eq_RR}
{\rm dim}_k \ {\rm H}^0\left(C,\mathcal{O}_C(na)\right)-
{\rm dim}_k \ {\rm H}^1(C,\mathcal{O}_C(na))
={\rm deg}(\mathcal{O}_C(na))+1-p_a(C).
\end{equation}

According to \cite[Th.~5.3.2]{QL}, if $n$ is large enough, then 
${\rm H}^1(C,\mathcal{O}_C(na))=0$. We make this assumption throughout this proof.

We denote $f:C\rightarrow \mathbb{P}$ the inclusion, and $\mathcal{I}_C$ the sheaf of 
ideal of $\mathcal{O}_{\mathbb{P}}$ that defines the closed subvariety $C$; then
\[
0 \rightarrow \mathcal{I}_C \rightarrow \mathcal{O}_\mathbb{P} \rightarrow
f_*\mathcal{O}_C \rightarrow 0
\]
is an exact sequence of sheaves on $\mathbb{P}$. Moreover
$\mathcal{I}_C \cong \mathcal{O}_\mathbb{P}(-d)= \mathcal{O}_\mathbb{P}(-ha)$, and the 
sheaf $\mathcal{O}_\mathbb{P}(na)$ is invertible (Lemma \ref{lem_genre}), so in particular 
flat. Then 
\begin{equation}\label{suite_C}
0 \rightarrow \mathcal{O}_\mathbb{P}(na-ha) \rightarrow \mathcal{O}_\mathbb{P}(na)
\rightarrow f_*\mathcal{O}_C(na) \rightarrow 0
\end{equation}
is an exact sequence of sheaves on $\mathbb{P}$.

As above, we can take $n$ large enough, so that
${{\rm H}^1\left(\mathbb{P},\mathcal{O}_\mathbb{P}(na-ha)\right)=0}$. Then the 
cohomological exact sequence induced by the sequence \eqref{suite_C} is
\[
0 \rightarrow {\rm H}^0\left(\mathbb{P},\mathcal{O}_\mathbb{P}(na-ha)\right)
\rightarrow {\rm H}^0\left(\mathbb{P},\mathcal{O}_\mathbb{P}(na)\right)
\rightarrow {\rm H}^0\left(C,\mathcal{O}_C(na)\right)
\rightarrow 0.
\]
Thus 
\begin{equation}\label{eq_H0}
{\rm dim}_k \ {\rm H}^0\left(C,\mathcal{O}_C(na)\right)={\rm dim}_k \
{\rm H}^0\left(\mathbb{P},\mathcal{O}_\mathbb{P}(na)\right)-
{\rm dim}_k \ {\rm H}^0\left(\mathbb{P},\mathcal{O}_\mathbb{P}(na-ah)\right).
\end{equation}

Next, we compute 
${\rm dim}_k \ {\rm H}^0\left(\mathbb{P},\mathcal{O}_\mathbb{P}(\delta a)\right)$. 
As in Lemma \ref{lem_genre}, let $S=\bigoplus\limits_{d \in \NN} S_d$ be the graded
$k$-algebra $k[x,y,z]$ with weight $1$ for $x$, $1$ for $y$ and $a$ for $z$.
According to \cite[Chap.~5 §5.1 Pro.~1]{bbk_alg_lie46},
${\rm dim}_k \ S_{\delta a}$ is the
$\delta a$-th coefficient of the formal series
\[
\frac{1}{(1-t)^2(1-t^a)}=\left(\sum_{l\geq 0}(l+1)t^l\right)
\left( \sum_{i\geq 0} t^{ia}\right).
\]

Then 
\begin{align*}
{\rm dim}_k \ {\rm H}^0\left(\mathbb{P},\mathcal{O}_\mathbb{P}(\delta a)\right)=
{\rm dim}_k \ S_{\delta a}=&
\sum_{l+ai=\delta a}l+1 \\
 =& \sum_{i=0}^\delta \delta a -ia+1\\
 =& (\delta a +1)(\delta +1)-a\frac{\delta(\delta+1)}{2} \\
 =& \frac{(\delta+1)(\delta a +2)}{2}.
\end{align*}

By combining the equations \eqref{eq_RR}, \eqref{eq_H0} and the equation above we obtain:
\begin{align*}
{\rm deg}\left(\mathcal{O}_C(na)\right)+1-p_a(C)=&\frac{(n+1)(na+2)}{2}-
\frac{(n-h+1)(na-ha+2)}{2} \\
=& nah+ \frac{2h+ah-ah^2}{2}.
\end{align*}

Finally:
\[
p_a(C)=1+ \frac{ah^2-2h-ah}{2}=\frac{(h-1)(ah-2)}{2}.
\]
\end{proof}

We are going to apply Proposition \ref{pro_genre_deg} to the study of the arithmetic 
genus of the regular completion of the forms of $\mathbb{G}_{a,k}$.
This genus has been studied by C.~Greither for a form $X$ of $\mathbb{A}^1_k$ in the 
particular case when the minimal field $k'$ such that $X_{k'}\cong \mathbb{A}^1_{k'}$
is of degree $p$ \cite[Th.~3.4]{haut1} and \cite[Th.~4.6]{haut1}.

\begin{Cor}\label{cor_deg_genre}
Let $k$ be a field of characteristic $p>0$, and $G$ be a form of $\mathbb{G}_{a,k}$.
We note $n=n(G)$ and $m$ the smallest integer such that $G$ is defined by
$y^{p^n}=x+a_1x^p+\dots+a_mx^{p^m}$. Let $C$ be the regular completion of $G$,
then 
\begin{equation}\label{genre_C}
p_a(C)\leq \frac{(p^{{\rm min}(n,m)}-1)(p^{{\rm max}(n,m)}-2)}{2}.
\end{equation}
Moreover, $a_m \notin k^p$ if and only if \eqref{genre_C} is an equality.
\end{Cor}

In order to show this Corollary we are going to introduce a "naive completion"
$\widehat{C}$ of $G$. The "naive completion" will give us a geometrical interpretation of 
the condition $a_m \notin k^p$: this is equivalent to $\widehat{C}$ being regular.

First we suppose that $n\leq m$.
Let $\widehat{C}$ be the closure of $G$ in $\mathbb{P}_k(1,p^{m-n},1)$, then
$\widehat{C}$ is defined by the homogeneous polynomial
\begin{equation}\label{eq_naif}
y^{p^n}-\left( xz^{p^n-1}+a_1x^pz^{p^n-p}+\dots+a_mx^{p^m}\right),
\end{equation}
where $x$ has weight $1$, $y$ has weight $p^{m-n}$, and $z$ has weight $1$.

Let $A$ be the graded $k$-algebra defined as the quotient of the graded algebra
$k[x,y,z]$ (with weights as above) by the ideal generated by the homogeneous polynomial 
\eqref{eq_naif}, then $\widehat{C}={\rm Proj}(A)$.

Let us consider the affine open $D_+(x)$ of $\mathbb{P}_k(1,p^{m-n},1)$, the affine 
variety ${\widehat{C}\cap D_+(x)}$ is the spectrum of $A_{(x)}$, the sub-algebra of
$A[\frac{1}{x}]$ of elements of degree $0$. Then $A_{(x)}$ is generated by
$y/x^{p^{m-n}}$ and $z/x$.

Let $Y=y/x^{p^{m-n}}$ and $Z=z/x$, then
\[
A_{(x)}=k[Y,Z]/\left(Y^{p^n}-\left(Z^{p^n-1}+a_1Z^{p^n-p}+\dots+a_m\right)\right).
\]

Also $\widehat{C}\setminus G$ is a unique point that we will note $\infty$. A 
straightforward computation shows that
\[
\frac{\mathcal{O}_{\widehat{C},\infty}}{(z)}\cong \frac{k[y]}{(y^{p^n}-a_m)}.
\]
If $a_m \notin k^p$ then $\frac{k[y]}{(y^{p^n}-a_m)}$ is a field, so $\widehat{C}$ is 
regular, thus $\widehat{C}$ is the regular completion $C$.

Let us consider the morphism $\widehat{C} \rightarrow \mathbb{P}^1_k$ induced by
the projection $p_x:G \rightarrow \mathbb{G}_{a,k}$. The scheme theoretic 
fibre of this morphism at $[1:0]$ is
${\rm Spec}\left(\mathcal{O}_{\widehat{C},\infty}/(z)\right)$ so 
$z$ is a uniformizing parameter of $\widehat{C}$ at $\infty$ if and only if
$\widehat{C}$ is regular if and only if $a_m \notin k^p$.

If $n > m$, the construction of the naive completion $\widehat{C}$ is almost the same, 
except that $\widehat{C}$ is the closure of $G$ in $\mathbb{P}_k(p^{n-m},1,1)$. The curve 
$\widehat{C}$ is defined by the homogeneous equation
\[
y^{p^n}=xz^{p^n-1}+a_1x^pz^{p^m-p}+\dots+a_mx^{p^m},
\]
where $x$ has weight $p^{n-m}$, $y$ has weight $1$, and $z$ has weight $1$.

And $\widehat{C}\setminus G$ is a unique point still denoted $\infty$, then
\[
\frac{\mathcal{O}_{\widehat{C},\infty}}{(z)}\cong \frac{k[x]}{(x^{p^m}-a_m^{-1})}.
\]
By the same argument as above $a_m \notin k^p$ if and only if $\widehat{C}$ is regular.

\begin{proof}
Assume $a_m \notin k^p$. Then
we have shown that $\widehat{C}$ is regular, so
(by unicity of the regular completion) $\widehat{C}$ is the regular 
completion $C$. And according to Proposition~\ref{pro_genre_deg}, we have
\[p_a(C)=\frac{1}{2}(p^{{\rm min}(n,m)}-1)(p^{{\rm max}(n,m)}-2).\]

On the other hand, if $a_m \in k^p$, then $\widehat{C}$ is not normal. Let
$\pi:C \rightarrow \widehat{C}$ 
be the normalisation. There is an exact sequence of sheaves on $\widehat{C}$:
\[
0 \rightarrow \mathcal{O}_{\widehat{C}} \rightarrow \pi_*\mathcal{O}_C \rightarrow 
\mathcal{F} \rightarrow 0
\]
where $\mathcal{F}$ is a non trivial sheaf with support $\infty$. So 
$p_a(C)=p_a(\widehat{C})-{\rm dim}_k H^0(\widehat{C},\mathcal{F})$, then 
\[p_a(C)< p_a(\widehat{C})=\frac{1}{2}(p^{{\rm min}(n,m)}-1)(p^{{\rm max}(n,m)}-2).\]
\end{proof}

\section[Picard group of the forms of the affine line]{Picard group of the forms of $\mathbb{A}^1_k$}
\label{sec_pic_grp}

\subsection{An exact sequence of Picard groups}\label{pic_X_C}

Let $X$ be a form of $\mathbb{A}^1_k$, in this Subsection we will link
the Picard group of $X$ to the Picard group of $C$ by adapting the argument of 
\cite[Th.~6.10.1]{KMT}. The curves $C$ and $X$ are regular, so we can identify 
the Picard group with the divisor class group, thus we note $[P_\infty]$ the 
class of the point $P_\infty$ in the Picard group of $C$. The following 
sequences are exact \cite[Pro.~II.6.5]{hart.ag}:
\begin{equation}\label{exacte.1}
0 \rightarrow \ZZ[P_\infty] \rightarrow {\rm Pic}(C) 
\rightarrow {\rm Pic}(X) \rightarrow 0,
\end{equation}
and
\begin{equation}\label{deg_C}
0 \rightarrow {\rm Pic}^0(C) \rightarrow {\rm Pic}(C) \rightarrow m(X)\ZZ
\rightarrow 0
\end{equation}
where $m(X)$ is the invariant of $X$ defined in \ref{def_m}.

By combining the two exact sequences \eqref{exacte.1} and \eqref{deg_C} we 
obtain the following exact sequence:
\begin{equation}\label{exacte.G.pic}
0 \rightarrow {\rm Pic}^0(C) \rightarrow {\rm Pic}(X) \rightarrow 
m(X)\ZZ/p^{r(X)}\ZZ 
\rightarrow 0.
\end{equation}

\begin{Exe}
As in Example \ref{C_P1}, let $p=2$ and let $G$ be the form of
$\mathbb{G}_{a,k}$ defined by the equation $y^2=x+ax^2$ where $a \notin k^2$. 
Since $C\cong \mathbb{P}^1_k$, we obtain ${\rm Pic}(G) \cong \ZZ/2\ZZ$.

More generally, let $Q$ be a purely inseparable point of $\mathbb{P}^1_k$, 
then $X=\mathbb{P}_k^1\setminus \{Q\}$ is a non trivial form of
$\mathbb{A}^1_k$ and ${\rm Pic}(X)\cong \ZZ/{\rm deg}(Q)\ZZ$.
\end{Exe}

\begin{Exe}
Let $k$ be a field of characteristic $p\neq 2$, let $G$ be the form of
$\mathbb{G}_{a,k}$ defined by the equation $y^p=x+ax^p$ where $a \notin k^p$.

Let $P_0$ be the neutral element of $G$, the morphism
$$P\in G(k) \mapsto [P]-[P_0] \in {\rm Pic}^0(C)$$
is injective \cite[Th.~6.7.9]{KMT},
so if $G(k)$ is infinite (e.g. if $k$ is separably closed), then
${\rm Pic}(G)$ is an infinite group (Recall that over a perfect field, the 
Picard group of an affine connected smooth algebraic group is finite
\cite[Lem.~6.9]{san_grp_br}).
\end{Exe}

\begin{Rem}
For every extension $K$ of $k$, there is a regular completion $C^K$ of $X_K$
which is not necessary the base change $C_K$ (if $K$ is not a separable 
extension of $k$, then $C_K$ can be no longer regular). So there is an exact 
sequence
\[
0 \rightarrow {\rm Pic}^0\left(C^K\right) \rightarrow {\rm Pic}(X_K) 
\rightarrow m(X_K)\ZZ/p^{r(X_K)}\ZZ \rightarrow 0.
\]
\end{Rem}

This motivates the study of $\textrm{ Pic}^0_{C/k}$ which is going to be done 
in Section~\ref{pic_fonct}. Before that in Section \ref{sec2} we are going to 
gather some results on the Picard functor that will be used in Section 
\ref{pic_fonct}.

\subsection{Proof of the main theorem}

\begin{The}\label{main_th}
Let $X$ be a non trivial form of $\mathbb{A}^1_k$. 
\begin{enumerate}[(i)]
\item ${\rm Pic}(X)$ is $p^{n(X)}$-torsion.

\item If $X$ has a $k$-rational point (e.g. $X$ is a form of
$\mathbb{G}_{a,k}$ or $k=k_s$), then ${{\rm Pic}(X)\neq \{0\}}$.
\end{enumerate}
\end{The}

\begin{proof}
$(i)$ Let $n$ be $n(X)$. The $n$th relative Frobenius morphism
$F^{n}_X:X \rightarrow X^{(p^n)}$ is a finite surjective map
of degree $p^{n}$, we will denote it $f$. Let $Z$ be a cycle of codimension
$1$ on $X$, then $f_*Z$ is a cycle of codimension $1$ on $X^{(p^n)}$
\cite[Cor.~8.2.6]{QL}. A direct consequence of the definition of $f$ is 
that $f$ is injective on topological spaces, so 
$f^*f_*Z={\rm deg}(f)Z=p^{n}Z$ \cite[Pro. 7.1.38]{QL}. Moreover
$X^{(p^n)}\cong \mathbb{A}^1_{k^{p^{-n}}}$, so
$f_*Z=0$ in ${\rm Pic}\left(X^{(p^n)}\right)$. Thus $f^*f_*Z=p^{n(X)}Z=0$ in
${\rm Pic}(X)$, and the group ${\rm Pic}(X)$ is of $p^{n(X)}$-torsion.

$(ii)$ If there is a $k$-rational point on $X$ then $m(X)=1$. By hypothesis
$X$ is a non trivial form of $\mathbb{A}^1_k$, so $P_\infty$ is a non
$k$-rational purely inseparable 
point (Remark~\ref{rem_rat}), and $\ZZ[P_\infty]$ is a strict subgroup of
${\rm Pic}(C)$. So ${\rm Pic}(X)$ is non trivial.
\end{proof}

We will now use the arguments of the proof of Theorem \ref{main_th} $(i)$ to 
obtain an upper bound on the torsion of other Picard groups.

Let $Y$ be an affine geometrically integral algebraic variety of dimension
$d$. First, remark that the definition of the $n$th relative Frobenius 
morphism stated in Subsection \ref{sec_intro} extends to the setting of every
affine $k$-scheme. So in particular, $F^n_Y: Y \rightarrow Y^{(p^n)}$ is well 
defined and is a finite morphism of degree $p^{dn}$. Next, let $n(Y)$ be the 
smallest non-negative integer $n$ such that
$Y^{(p^n)} \cong \mathbb{A}_{k^{p^{-n}}}^d$ (if it exists). This notation 
coincides with that of Subsection \ref{sec_intro} if $Y$ is a form of
$\mathbb{A}^1_k$.

\begin{Lem}\label{lem_n(Y)}
The integer $n(Y)$ is well defined in the following cases:
\begin{enumerate} [(i)]
\item $Y$ is a smooth connected unipotent algebraic group.

\item $Y$ is a form of $\mathbb{A}^2_k$.
\end{enumerate}
\end{Lem}

\begin{proof}
For $(i)$ see \cite[Cor.~IV § 2 3.9]{DG} and
\cite[Th.~IV §4 4.1]{DG}. For $(ii)$ see \cite[Th. 3]{A2PI}.
\end{proof}

The following proposition is obtained by arguing as in the proof of 
Theorem~\ref{main_th}~$(i)$.

\begin{Pro}\label{main_th_gen}
\begin{enumerate}[(i)]
\item Let $U$ be a smooth connected unipotent algebraic group, let $d$ be the 
dimension of $U$. Then ${\rm Pic}(U)$ is of $p^{dn(U)}$-torsion.

\item Let $Y$ be a form of $\mathbb{A}^2_k$. Then ${\rm Pic}(Y)$ is of
$p^{2n(Y)}$-torsion.

\item Let $k$ be separably closed, let $d\in \NN^*$, and let $Y$ be a form of 
$\mathbb{A}^d_{k}$. Then ${\rm Pic}(Y)$ is of $p^{dn(Y)}$-torsion. 
\end{enumerate}
\end{Pro}

\begin{Rem}
Let $d\geq 3$, and let $Y$ be a form of $\mathbb{A}^d_k$. It is not known if 
there is a purely inseparable extension $k'/k$ such that
$Y_{k'} \cong \mathbb{A}^d_{k'}$.
\end{Rem}

\subsection{Examples of forms of the affine line with trivial Picard group}
\label{pic_triv}

First, we give an explicit example of a form of $\mathbb{A}^1_{k}$ with 
trivial Picard group.

\begin{Lem}\label{lem_pic_trivial}
Let $k=\mathbb{F}_2(t,u)$ and let $X$ be the form of $\mathbb{A}^1_k$ defined 
by the equation ${y^2=u+x+tx^2}$, then $X(k)=\emptyset$ and ${\rm Pic}(X)$ is 
trivial.
\end{Lem}

\begin{proof}
In order to show that $X(k)=\emptyset$, it is enough to show that the 
only solution of $P^2=uQ^2+QR+tR^2$, where ${P, Q, R \in \mathbb{F}_2[t,u]}$ 
is trivial.
We denote ${\rm deg}$ the total degree of a polynomial.
If ${\rm deg}(Q)\neq {\rm deg}(R)$, for example ${\rm deg}(Q)<{\rm deg}(R)$, 
then
\[
{\rm deg}(P^2)={\rm deg}(uQ^2+QR+tR^2)=1+2{\rm deg}(R).
\]
So ${\rm deg}(P^2)$ is odd, contradiction. So ${\rm deg}(Q)={\rm deg}(R)$ and 
the monomials of highest degree of $uQ^2$ and $tR^2$ must cancel (if they 
don't cancel we have the same contradiction). But this is impossible because 
the monomial of highest degree of $uQ^2$ has an odd partial degree in $u$ 
whereas the monomial of highest
degree of $tR^2$ has an even partial degree in $u$. 

After a field extension to $k_s$, by Proposition \ref{homo} $X_{k_s}$ is 
isomorphic, as a scheme, to the non trivial form of $\mathbb{G}_{a,k_s}$ of 
equation $y^2=x+tx^2$. We have 
seen in Example~\ref{C_P1} that the regular completion of this form of
$\mathbb{G}_{a,k_s}$ 
is $\mathbb{P}^1_{k_s}$, so by uniqueness of the regular completion
$C_{k_s}\cong \mathbb{P}^1_{k_s}$. Then ${\rm Pic}^0_{C_{k_s}/k_s}$ is 
trivial, and ${\rm Pic}^0_{C/k}$ too. By \cite[Th.~9.3.1]{NM},
${\rm Pic}^0(C)$ is a subgroup of ${\rm Pic}^0_{C/k}(k)$, hence trivial.

Moreover $X$ has no $k$-rational point so according to a theorem by T. A. 
Springer \cite[Cor.~18.5]{Spr_char2} $X$ has no rational point on any 
extension of odd degree. Thus
${\rm deg}: \ZZ[P_\infty] \overset{\sim}{\rightarrow} {\rm Pic}(C)$, and by 
exactness of the sequence \eqref{exacte.1}, ${\rm Pic}(X)$ is trivial.
\end{proof}

\begin{Rem}
The regular completion of the form consider in Lemma \ref{lem_pic_trivial} is
a non trivial form of $\mathbb{P}^1_k$.
\end{Rem}

We will now construct a family of forms of the affine line with trivial Picard group.
Let $G$ be a non trivial form of $\mathbb{G}_{a,k}$, by Theorem \ref{rus.2.1} there is an 
exact sequence:
\[
0 \rightarrow G \rightarrow \mathbb{G}_{a,k}^2 \rightarrow \mathbb{G}_{a,k} \rightarrow 
0.
\]
Let $\eta$ be the generic point of $\mathbb{G}_{a,k}$, then $\eta={\rm Spec}(k(t))$ and
comes with a map $\eta \rightarrow \mathbb{G}_{a,k}$.
We denote $X$ the fibre product
$\mathbb{G}_{a,k}^2\times_{\mathbb{G}_{a,k}} \eta$.

\begin{Pro}\label{pro_triv}
With the above notation $X$ is a non trivial form 
of $\mathbb{A}^1_{k(t)}$ and
${\rm Pic}\left(X \right)$ is trivial.
\end{Pro}

\begin{proof}
The morphism $\mathbb{G}_{a,k}^2 \rightarrow \mathbb{G}_{a,k}$ is a $G$-torsor. So
$X \rightarrow \eta$ is a $G_{k(t)}$-torsor, and in particular a form of
$\mathbb{A}^1_{k(t)}$.
Let $K$ be a separable closure of $k(t)$, then $G_{K}$ is still $K$-wound ($K/k(t)$ and 
$k(t)/k$ are separable extensions, and being wound is not changed by separable extension 
\cite[B.3.2]{CGP}).
By definition of $k$-wound, $G_{K}$ is a non trivial form of $\mathbb{A}^1_K$.
Moreover $X_K$ is an homogeneous space under $G_{K}$, so $X_K$ is a non trivial form of
$\mathbb{A}^1_K$ and in particular $X$ is a non trivial form of
$\mathbb{A}^1_{k(t)}$.

Finally, at the algebraic level the morphism
$X \rightarrow \mathbb{G}_{a,k}^2$ is the localisation morphism 
\[
k[x,y] \rightarrow k[x,y]\otimes_{k[T]} k(T),
\]
where $T=y^{p^n}-\left(x+a_1x^p + \dots + a_mx^{p^m} \right)$ is a polynomial that defines 
$G$. Then ${\rm Pic}(\mathbb{G}_{a,k}^2) \twoheadrightarrow
{\rm Pic}(X)$ \cite[Chap. 7 §1 n°10 Pro.~17]{algcom57}, thus
${\rm Pic}\left( X\right)$ is trivial.
\end{proof}

\begin{Rem}
Let $k$ be an imperfect field, with the construction of Proposition \ref{pro_triv} we have 
an example of a non trivial form of $\mathbb{A}^1_{k(t)}$ with trivial Picard group.
\end{Rem}

Let $G$ be a smooth affine algebraic group, we recall that $G$ is said to be
{\it special} if for any field extension $K$ of $k$, any $G_K$-torsor
$X\rightarrow {\rm Spec}(K)$ is trivial.

J.-P. Serre initiated the study of special groups over an 
algebraically closed field in \cite{serre_EFA} and
A. Grothendieck classified these groups  \cite{Gro_thsr}. More recently J.-L.
Colliot-Thélène and J.-J. Sansuc characterised special tori over an  arbitrary 
field \cite[Pro.~7.4]{tore_spe}, and M. Huruguen characterised the special 
reductive groups over an arbitrary field \cite[Th.~4.1]{huru}.
It is known that $\mathbb{G}_{a,k}$ is special, by the arguments of
\cite[4.4.a]{serre_EFA}, and more 
generally that every smooth connected $k$-split unipotent algebraic group is 
special. D. T. Nguyễn showed, under a mild assumption on the base field, that 
a smooth unipotent algebraic group is special if and only if it is $k$-split
\cite[Cor.~6.10]{ess_dim_unip}. In an unpublished note he generalised the 
result to an arbitrary base field \cite{note_ngu}.
We are going to show this result, in the particular case of 
the forms of $\mathbb{G}_{a,k}$, by using a different method; we see it as a 
corollary of our main Theorem \ref{main_th} and Proposition \ref{pro_triv}:

\begin{Cor}\label{cor_spe}
Let $G$ be a non trivial form of $\mathbb{G}_{a,k}$, then the
$G_{k(t)}$-torsor ${X\rightarrow {\rm Spec}(k(t))}$ of Proposition 
\ref{pro_triv} is non trivial, thus $G$ is not special.
\end{Cor}

\begin{proof}
Assume that ${X\rightarrow {\rm Spec}(k(t))}$ is a trivial $G_{k(t)}$-torsor, 
then in particular ${{\rm Pic}(X)\cong {\rm Pic}(G)}$. But this is impossible 
since ${\rm Pic}(G)$ is not trivial (Theorem~\ref{main_th}), while
${\rm Pic}(X)$ is trivial (Proposition \ref{pro_triv}).
\end{proof}

\section{Cocartesian diagram and Picard functor}\label{sec2}
The main result of this Section is Theorem \ref{th_pic_final}, it is stated 
and proved in Subsection~\ref{sub_pic_final}. In Subsection \ref{rem_V_mu} we 
gather some auxiliary results on the unit group scheme, in Subsection 
\ref{sub_rig} we show Proposition \ref{pro_loc_alg} that is the main tool for 
the proof of Theorem \ref{th_pic_final}.

Throughout this Section $S$ is a base scheme, we consider
schemes and morphisms over $S$. And if $X$ and $T$ are two $S$-schemes we will 
note $X_T$ for the product $X\times_S T$.

We will use this level of generality, in a future work, to study the
$G$-torsors for $G$ a form of $\mathbb{G}_{a,k}$.

\subsection{Unit group scheme}\label{rem_V_mu}

Let $f:X \rightarrow S$ be a proper morphism, flat and of finite presentation. The 
functor 
\[
\begin{array}{rcl}
 Sch/S^\circ &\rightarrow &Rings \\
 T& \mapsto & \mathcal{O}(X_T)
\end{array}
\]
is represented by a $S$-scheme $V_X$ which is smooth if and only if $f$ is 
cohomologically flat in dimension~$0$ \cite[Cor.~8.1.8]{NM} (i.e. the formation of 
$f_*(\mathcal{O}_X)$ commutes with base change). Moreover the functor 
\[
\begin{array}{rcl}
 Sch/S^\circ &\rightarrow &Groups \\
 T& \mapsto & \mathcal{O}(X_T)^*
\end{array}
\]
is represented by an open sub-scheme $\mu^X$ of $V_X$, so it is an $S$-group scheme
\cite[Lem.~8.1.10]{NM} .

An important particular case is the following, let $k$ be a field of arbitrary 
characteristic, let $A$ be a $k$-algebra of finite dimension (as a $k$-vector space), 
then the group functor
\[
\begin{array}{rcl}
k-algebras &\rightarrow &Groups \\
R& \mapsto & (A\otimes_k R)^*
\end{array}
\]
is represented by an affine smooth commutative 
connected algebraic group denoted $\mu^A$ whose Lie algebra is $A$ with the trivial 
bracket \cite[II §1 2.3]{DG}.

\begin{Rem}
Let $k$ be a field of arbitrary characteristic, then
$\mu^k$ is the multiplicative group $\mathbb{G}_{m,k}$.

More generally if $A$ is a $k$-algebra of finite dimension, then
$\mu^A=R_{A/k}(\mathbb{G}_{m,A})$, where $R_{A/k}$ is the Weil restriction.
\end{Rem}

\begin{Rem}
Let $A \subset A'$ be two $k$-algebras of finite dimension. The inclusion
$f:A \hookrightarrow A'$ induces a morphism of algebraic groups
$f^*:\mu^{A} \rightarrow \mu^{A'}$ which is injective on 
$\overline{k}$-rational points and induces an injection on the Lie algebras. 
So the scheme theoretic kernel of $f^*$ is trivial, and  $f^*$ is a closed 
immersion \cite[II~§5~5.1]{DG}.

The co-kernel of $f^*$ is a smooth commutative connected affine algebraic 
group denoted $\mu^{A'/A}$.
\end{Rem}

\begin{Lem}\label{lem_dep}
Let $k$ be a field of arbitrary characteristic.

{\rm (i)} Let $A$ be a local $k$-algebra, of finite dimension. 
Let $M$ be the maximal ideal of $A$ and $K$ the residue field of $A$. We have 
an exact sequence of algebraic groups
\[
0\rightarrow 1+M \rightarrow \mu^A \rightarrow \mu^K \rightarrow 0
\]
where $1+M$ is a $k$-split smooth connected unipotent algebraic group. 

Moreover if the residue field $K$ is $k$, this sequence has a unique splitting 
and we have a canonical isomorphism ${\mu^A\cong(1+M)\times_k \mu^k}$.

{\rm (ii)} Let $A\subset A'$ be two local $k$-algebras, of finite dimension 
and having the same residue field $K$, then $\mu^{A'/A}$ is a $k$-split smooth 
connected unipotent group.
\end{Lem}

\begin{proof}
$(i)$ First we look at the composition series associated to the 
$k$-sub-algebras $k\oplus M^n$, the successive quotients are vector groups 
associated with the $k$-vector spaces $M^n/M^{n+1}$. So $1+M$ is a $k$-split 
unipotent algebraic group.

The quotient map $p:A\twoheadrightarrow A/M\cong K$ induces a morphism of 
algebraic groups $\mu^A \rightarrow \mu^K$, then
$\mu^A(\overline{k}) \twoheadrightarrow \mu^K(\overline{k})$,
so $\mu^A \rightarrow \mu^K \rightarrow 0$ is exact and the kernel of
$\mu^A \rightarrow \mu^K$ is $1+M$.

Moreover if $K=k$, then $A=k\oplus M$ and the inclusion $k\subset A$ is the 
unique morphism $k\rightarrow A$. So there is a unique section of
the morphism $\mu^A \rightarrow \mu^k$, and
${\mu^A\cong(1+M)\times_k \mu^K}$ canonically.

$(ii)$ According to $(i)$, the rows of the commutative diagram below are 
exact.
\[
\xymatrix{
0\ar[r]& 1+M \ar[r] \ar[d]& \mu^A \ar[r] \ar[d]& \mu^K \ar[r] \ar[d]& 0 \\
0\ar[r]& 1+M' \ar[r] & \mu^{A'} \ar[r] & \mu^K \ar[r]& 0.
}
\]
So there is an isomorphism $\mu^{A'/A}=\mu^{A'}/\mu^{A}\cong (1+M')/(1+M)$.
In particular $\mu^{A'/A}$ is a $k$-split unipotent group (as a quotient of a 
split unipotent group \cite[Th.~V.15.4]{bor.lag}).
\end{proof}

\begin{Rem}\label{rem_ploye}
Let $k$ be a field of positive characteristic.

Let $A \subset A'$ be two local $k$-algebras of finite dimension, with residue 
fields $K$ and $K'$, then $\mu^{A'/A}$ is not necessary $k$-split. For 
example, if $A=K=k$ and $A'=K'$ is a purely inseparable extension of finite 
degree of $k$, then according to \cite[Lem.~VI.5.1]{Oe} $\mu^{K/k}$ is
$k$-wound.
\end{Rem}

\subsection{Rigidified Picard functors}\label{sub_rig}

The main result of this Subsection is the exact sequence 
\eqref{suite_pic_rigid} which relates the Picard functor and the rigidified 
Picard functor.

Let $X\rightarrow S$ be a proper, flat morphism of finite presentation. 

\begin{Def}
Following \cite{NM} we will define the rigidified Picard functor.
First we define a sub-scheme
$Y \subset X$ which is finite, flat, and of finite presentation over S,
to be a rigidificator (also called rigidifier) of ${\rm Pic}_{X/S}$ if for all 
$S$-schemes $T$ the map $\mathcal{O}(X_T) \rightarrow \mathcal{O}(Y_T)$ 
induced by the inclusion of schemes $Y_T \rightarrow X_T$ is injective.

Let $Y$ be a rigidificator of ${\rm Pic}_{X/S}$, a rigidified line bundle on
$X$ along $Y$ is by definition a pair $(\mathcal{L},\alpha)$ where
$\mathcal{L}$ is a line bundle on $X$ and $\alpha$ is a isomorphism
$\mathcal{O}_Y \overset{\sim}{\rightarrow} \mathcal{L}_{|Y}$.

Let $(\mathcal{L},\alpha)$ and $(\mathcal{L}',\alpha')$ be two rigidified line 
bundle on $X$ along $Y$. A morphism of rigidified line bundle
$f:(\mathcal{L},\alpha) \rightarrow (\mathcal{L}',\alpha')$ is a morphism of 
line bundle $f:\mathcal{L} \rightarrow \mathcal{L}'$ such that
$f_{|Y}\circ\alpha=\alpha'$.

We can now define the rigidified Picard functor as the functor
\[
({\rm Pic}_{X/S}, Y): (Sch/S)^0 \rightarrow  (Set)
\]
which associates to the $S$-scheme $T$ the set of isomorphisms of rigidified 
line bundles on $X_T$ along $Y_T$.
\end{Def}

There is a map
\[\begin{array}{rccl}
\delta:&\mu^Y & \rightarrow & ({\rm Pic}_{X/S}, Y) \\
&a\in\mathcal{O}(Y\times_S T)^* & \mapsto & (\mathcal{O}_{X\times_S T},
{\rm mult_a})
\end{array}\]
where the map
${\rm mult_a}: \mathcal{O}_{X\times_S T} \overset{\sim}{\rightarrow}
\mathcal{O}_{X\times_S T}$ is the multiplication by
${a\in\mathcal{O}(Y\times_S T)^*}$.
There is also a map $({\rm Pic}_{X/S}, Y) \rightarrow {\rm Pic}_{X/S}$
which forgets the rigidification and whose kernel is the image of $\delta$.

According to \cite[Pro.~2.1.2]{Ray_spe_pic} and
\cite[Pro.~2.4.1]{Ray_spe_pic}, the sequence
\[
0\rightarrow \mu^X \rightarrow \mu^Y \rightarrow ({\rm Pic}_{X/S}, Y) \rightarrow 
{\rm Pic}_{X/S} \rightarrow 0
\]
is an exact sequence of sheaves for the étale topology.

Under the above hypotheses we can apply \cite[Th.~2.3.1]{Ray_spe_pic}, so the
rigidified Picard functor $({\rm Pic}_{X/S}, Y)$ is represented by an algebraic space of
finite presentation on $S$. 

In Remark \ref{rem_loc_alg} and in Proposition \ref{pro_loc_alg}  we will present particular cases where $({\rm Pic}_{X/S}, Y)$ is represented by an $S$-group scheme.

\begin{Rem}\label{rem_loc_alg}
Let $X \rightarrow S$ be cohomologically flat in dimension $0$, then 
${\rm Pic}_{X/S}$ is represented by an $S$ group scheme locally of finite type.
Moreover if $S$ is a field, then $({\rm Pic}_{X/S}, Y)$ is represented by a $S$-group 
scheme locally of finite type \cite[Lem.~4.2]{Artin}.
\end{Rem}

\begin{Pro}\label{pro_loc_alg}
Let $X\rightarrow S$ be a projective flat morphism of finite presentation, with
geometrically integral  fibres and let $Y\subset X$ be a rigidificator. Then,
\begin{enumerate}[(i)]
\item The quotient $\mu^Y/\mu^X$ is represented by an affine, flat $S$-group scheme of 
finite presentation.

\item The functor $({\rm Pic}_{X/S}, Y)$ is represented by an $S$-group scheme, locally
of finite presentation.

\item The sequence
\begin{equation}
\label{suite_pic_rigid}
0\rightarrow \mu^Y/\mu^X \rightarrow ({\rm Pic}_{X/S}, Y) \rightarrow 
{\rm Pic}_{X/S} \rightarrow 0
\end{equation}
is an exact sequence of $S$-group schemes, locally of finite presentation.
\end{enumerate}
\end{Pro}

\begin{proof}
The Picard functor ${\rm Pic}_{X/S}$ is represented
by a separated $S$-scheme locally of finite presentation \cite[Th.~8.2.1]{NM}.
Moreover $\mu^X=\mathbb{G}_{m,S}$ and $\mu^Y=R_{Y/S}(\mathbb{G}_{m,Y})$, so the
$S$-group scheme $\mu^Y$ is affine \cite[Pro.~I §16.6]{DG}. According to
\cite[Th.~VIII.5.1]{SGAIII_2}, the quotient
$\mu^{Y}/\mu^{X}$ is an affine $S$-scheme ($\mu^X \rightarrow \mu^Y$ is an immersion
\cite[Pro.~8.1.9]{NM}, so $\mathbb{G}_{m,S}$ acts freely on $\mu^Y$). In addition $\mu^Y$ 
is smooth and of finite presentation over $S$ \cite[Pro.~7.6.5]{NM}, so $\mu^Y/\mu^X$ is 
according to \cite[Pro.~8.5.8]{SGAIII_2} of finite presentation on $S$ and according to 
\cite[Cor.~2.2.11 (ii)]{EGAIV_2} $\mu^Y/\mu^X\rightarrow S$ is faithfully flat.

We are going to show that $\mu^{Y}/\mu^{X}$ is an $S$-group scheme.
Let $m':\mu^Y\times \mu^Y \rightarrow \mu^Y$ be the multiplication and
$p:\mu^Y \rightarrow \mu^Y/\mu^X$ be the quotient.
Then we have a morphism ${p\circ m': \mu^Y\times \mu^Y \rightarrow \mu^Y/\mu^X}$ which is
$\mu^X\times \mu^X$-invariant. So according to \cite[Th.~VIII.5.1]{SGAIII_2},
the quotient $(\mu^Y\times \mu^Y)/(\mu^X\times \mu^X)$ exists. By the universal property
of the categorical quotient (the torsors are categorical quotients \cite[Pro.~0.1]{MFK}),
there is a unique morphism $m$ such that the diagram
\[\xymatrix{
\mu^Y\times \mu^Y  \ar[r]^-{p\times p} \ar[rd]_{p \circ m'}&
(\mu^Y\times \mu^Y)/(\mu^X\times \mu^X) \ar[d]^{m}\\
& \mu^Y/\mu^X
}\]
is commutative. Moreover 
$(\mu^Y\times \mu^Y)/(\mu^X\times \mu^X)=\mu^Y/\mu^X\times \mu^Y/\mu^X$, so we 
have shown that there is a morphism
${m: \mu^Y/\mu^X\times \mu^Y/\mu^X \rightarrow \mu^Y/\mu^X}$.
Likewise, there are two morphisms ${e:S \rightarrow \mu^Y/\mu^X}$ and
$i:\mu^Y/\mu^X\rightarrow \mu^Y/\mu^X$. We only have to remark that by the 
universal property of quotients, the diagrams
\[\xymatrix{
\mu^Y/\mu^X \times \mu^Y/\mu^X \times \mu^Y/\mu^X \ar[r]^-{m\times id}
\ar[d]_{id\times m}
& \mu^Y/\mu^X \times \mu^Y/\mu^X \ar[d]^{m} \\
\mu^Y/\mu^X \times \mu^Y/\mu^X  \ar[r]^-{m} & \mu^Y/\mu^X,
}\]
\[\xymatrix{
\mu^Y/\mu^X  \ar[r]^-{e\times id} \ar[rd]_{id} & \mu^Y/\mu^X \times \mu^Y/\mu^X
 \ar[d]^{m} & \mu^Y/\mu^X  \ar[l]_-{id \times e} \ar[ld]^{id}\\
& \mu^Y/\mu^X
}\]
and
\[\xymatrix{
\mu^Y/\mu^X  \ar[r]^-{id\times i} \ar[rd]_{e \circ f} & \mu^Y/\mu^X \times \mu^Y/\mu^X
 \ar[d]^{m} & \mu^Y/\mu^X  \ar[l]_-{i \times id} \ar[ld]^{e \circ f}\\
& \mu^Y/\mu^X
}\]
(where $f$ is the structural morphism of $\mu^Y/\mu^X$) are commutative.
Thus we have shown $(i)$.

Let us show $(ii)$. The morphism $({\rm Pic}_{X/S}, Y) \rightarrow {\rm Pic}_{X/S}$
is a $\mu^{Y}/\mu^{X}$-torsor \cite[Cor.~III §4 1.8]{DG}. 
The $S$-group $\mu^{Y}/\mu^{X}$ is affine, so by \cite[Pro. III §4 1.9 a)]{DG}
$({\rm Pic}_{X/S}, Y)$ is represented by a $S$-scheme.
Moreover recall that $({\rm Pic}_{X/S}, Y)$ is an 
algebraic space \cite[Th.~2.3.1]{Ray_spe_pic}. A consequence of \cite[Pro.~8.3.5]{NM}
is that the morphisms of algebraic spaces between two schemes are exactly the morphisms
of schemes. 

There is only $(iii)$ left, according to \cite[Cor.~III §4 1.7]{DG} and
\cite[Cor.~III §1 2.11]{DG} if the morphism 
$({\rm Pic}_{X/S}, Y)\times_{{\rm Pic}_{X/S}} ({\rm Pic}_{X/S}, Y) \rightarrow
({\rm Pic}_{X/S}, Y)$ is faithfully flat of finite presentation, then
$({\rm Pic}_{X/S}, Y) \rightarrow {\rm Pic}_{X/S}$ has the same property. 
According to \cite[III §1 2.4]{DG}, 
\[
({\rm Pic}_{X/S}, Y)\times_{{\rm Pic}_{X/S}} ({\rm Pic}_{X/S}, Y) \cong
({\rm Pic}_{X/S}, Y)\times_S \mu^{Y}/\mu^{X}.
\]
Moreover we have already shown that
$\mu^{Y}/\mu^{X}\rightarrow S$ is faithfully flat of finite presentation, so
\[
({\rm Pic}_{X/S}, Y) \rightarrow {\rm Pic}_{X/S}
\]
is also faithfully flat of finite 
presentation.

To conclude we remark that
\[
({\rm Pic}_{X/S}, Y)\times_{{\rm Pic}_{X/S}} ({\rm Pic}_{X/S}, Y) \cong
({\rm Pic}_{X/S}, Y)\times_S \mu^{Y}/\mu^{X}
\]
so $\mu^Y/\mu^X$ is the kernel of
$({\rm Pic}_{X/S}, Y) \rightarrow {\rm Pic}_{X/S}$, hence $(iii)$.
\end{proof}

\subsection{An exact sequence of Picard schemes}\label{sub_pic_final}

\begin{The}\label{th_pic_final}
Let
\[
\xymatrix{
Y' \ar[r]^v \ar[d]_g & X' \ar[d]^f \\
Y \ar[r]^u & X
}
\]
be a commutative square of $S$-schemes, cocartesian in the category of ringed spaces. We make 
the following hypotheses:
\begin{enumerate}[(i)]
\item The morphisms $u$ and $v$ are closed immersions, the morphisms $g$ and $f$ 
are affine.
\item The structural morphisms $X \rightarrow S$ and $X'\rightarrow S$ are projective, 
flat of finite presentation with geometrically integral fibres .
\item $Y$ is a rigidificator of ${\rm Pic}_{X/S}$, and likewise $Y'$ is a rigidificator of
${\rm Pic}_{X'/S}$.
\end{enumerate}

Then the sequence
\begin{equation} \label{suite_pic_final}
0 \rightarrow \mu^Y \rightarrow \mu^{Y'} \rightarrow {\rm Pic}_{X/S} \rightarrow 
{\rm Pic}_{X'/S} \rightarrow 0
\end{equation}
is an exact sequence of $S$-group schemes locally of finite presentation.
\end{The}

\begin{proof}
According to Proposition \ref{pro_loc_alg}, the commutative diagram
\[\xymatrix{
0\ar[r]& \mu^X \ar[d] \ar[r] &\mu^Y \ar[d] \ar[r]& ({\rm Pic}_{X/S}, Y) \ar[d]^{f^*} 
\ar[r]& {\rm Pic}_{X/S} \ar[r]\ar[d]& 0\\
0\ar[r]& \mu^{X'} \ar[r] &\mu^{Y'} \ar[r]& ({\rm Pic}_{X'/S}, Y') \ar[r]&
{\rm Pic}_{X'/S} \ar[r]& 0
}\]
is a diagram of $S$-group schemes with exact lines.
By \cite[Lem.~2.2]{MB_pic_var} $f^*$ is an isomorphism, and 
$\mu^X\cong \mu^{X'} \cong \mathbb{G}_{m,S}$. So by diagram chasing the sequence
\eqref{suite_pic_final} is exact.
\end{proof}

\section[Picard functor of the regular completion]{Picard functor of the regular completion}
\label{pic_fonct}
\subsection{Torsion of the Picard functor}\label{tor_pic_funct}

Let $X$ be a form of $\mathbb{A}^1_k$, let $C$ be the regular completion of $X$. Let $K$ 
be a field such that the regular completion of $X_K$ is $\mathbb{P}^1_K$ (e.g. $K=k'$ or 
$K=k^{p^{-n'(X)}}$, $n'(X)$ being the integer defined in \ref{def_n'}). The base change
$C_K$ is not necessary normal, but the normalisation of $C_K$ is $\mathbb{P}^1_K$ because
it is the regular completion of $X_K$, and the regular completion is unique up to unique
isomorphism. Let $\pi: \mathbb{P}^1_K \rightarrow C_K$ be the normalisation. Following 
\cite{Fer} we show how $C_K$ is obtained from $\mathbb{P}^1_K$ via "pinching".

Let $\mathcal{C}$ be the conductor of $\mathcal{O}_{C_{K}}$ in
$\mathcal{O}_{\mathbb{P}^1_K}$, i.e. the sheaf of ideals of
$\pi_*\mathcal{O}_{\mathbb{P}^1_{K}}$
given by:
\[\mathcal{C}(U)=\left\lbrace
a \in \mathcal{O}_{\mathbb{P}^1_K}(\pi^{-1}(U)) \ | \ 
a.\mathcal{O}_{\mathbb{P}^1_K}(\pi^{-1}(U)) \subset 
\mathcal{O}_{C_{K}}(U)
\right\rbrace\]
for any open sub-scheme $U$ of $C_K$.

Then $\mathcal{C}$ is also a sheaf of ideals of $\mathcal{O}_{C_{K}}$.
Let $Y^K$ be the closed sub-scheme of $C_K$ associated to the sheaf of ideals
$\mathcal{C}$. Then $C_K$ is regular outside of $P_\infty$, so $\pi$ induces an 
isomorphism between $C_K\setminus P_\infty$ and $\mathbb{P}^1_K \setminus \infty$ (where 
$P_\infty$ is the unique point of $C_K\setminus X_K$, and $\infty$ is 
the unique point of $\mathbb{P}^1_K$ above $P_\infty$). So as a set, $Y^K$ is the 
point $P_\infty$ and by construction there is a closed immersion $Y^K\rightarrow C_{K}$. 
Finally let $Z^K$ be the fibre product $Y^K \times_{C_{K}} \mathbb{P}^1_{K}$.

We have obtained a commutative diagram of $K$-varieties:
\begin{equation}\label{diag_gen}
\xymatrix{
Z^K \ar[r] \ar[d]& \mathbb{P}^1_{K} \ar[d]^{\pi} \\
Y^K \ar[r] & C_{K}.
}
\end{equation}
By construction the diagram \eqref{diag_gen} is cartesian, in fact according to the 
scholium \cite[4.3]{Fer} the diagram is also cocartesian.

First we will explicit $Y^K$ and $Z^K$. The morphism $\pi$ induces a morphism of local 
rings $\pi^{\#}_{P_\infty}:\mathcal{O}_{C_K, P_{\infty}} \rightarrow 
\mathcal{O}_{\mathbb{P}^1_{K}, \infty}$ which is the normalisation.
Let $\mathfrak{C}$ be the conductor of $\mathcal{O}_{C_K, P_{\infty}}$ in
$\mathcal{O}_{\mathbb{P}^1_{K}, \infty}$, i.e.
\[
\mathfrak{C}=\left\{ x\in \mathcal{O}_{\mathbb{P}^1_{K}, \infty} \ | \ 
x.\mathcal{O}_{\mathbb{P}^1_{K}, \infty} \subset \mathcal{O}_{C_K, P_{\infty}} \right\}=\mathcal{C}_{P_\infty}.
\]
we then have explicitly
$Z^K={\rm Spec}(\mathcal{O}_{\mathbb{P}^1_{K}, \infty}/\mathfrak{C})$ and
$Y^K={\rm Spec}(\mathcal{O}_{C_K, P_{\infty}}/\mathfrak{C})$.

By construction the cocartesian diagram \eqref{diag_gen} satisfies the hypotheses of 
Theorem \ref{th_pic_final}. Thus we have an exact sequence of locally algebraic groups
over $K$:
\[
0 \rightarrow \mu^{Z^K/Y^K} \rightarrow {\rm Pic}_{C_{K}/K} 
\rightarrow {\rm Pic}_{\mathbb{P}^1_{K}/K}\rightarrow 0.
\]
The neutral component of ${\rm Pic}_{\mathbb{P}^1_{K}/K}$ is trivial and $\mu^{Z^K/Y^K}$ 
is connected. So we have an isomorphism of algebraic groups over $K$:
\[
\textrm{ Pic}^0_{C_{K}/K}\cong \mu^{Z^K/Y^K}.
\]
In particular $\textrm{ Pic}^0_{C_{K}/K}$ is smooth.

\begin{Rem}
If $K=k'$, then according to Lemma \ref{lem_dep}, the algebraic group $\mu^{Z^K/Y^K}$ is 
$k'$-split unipotent, so ${\rm Pic}^0_{C_{k'}/k'}$ is $k'$-split unipotent.
\end{Rem}

And if we look at points over $\overline{k}$ we have the following isomorphisms:
\begin{equation}\label{eq_mu}
\textrm{ Pic}^0_{C_{K}/K}(\overline{k})\cong \mu^{Z^K/Y^K}(\overline{k})
=\frac{\mu^{Z^K}(\overline{k})}{\mu^{Y^K}(\overline{k})}
=\dfrac{\left(\overline{k} \otimes_k \frac{\mathcal{O}_{\mathbb{P}^1_{K}, \infty}}
{\mathfrak{C}}\right)^*}
{\left(\overline{k} \otimes_k \frac{\mathcal{O}_{C_K, P_\infty}}{\mathfrak{C}}\right)^*}.
\end{equation}

\begin{Lem}If $K=k^{p^{-n'(X)}}$, then the algebraic group $\mu^{Z^K/Y^K}$ is of
$p^{n'(X)}$-torsion.
\end{Lem}

\begin{proof}
$\mu^{Z^K/Y^K}$ is a smooth algebraic group, so it is enough to show that the 
group of $\overline{k}$-points $\mu^{Z^K/Y^K}(\overline{k})$ is of
$p^{n'(X)}$-torsion.

Let $n$ be a non-negative integer, then
$\kappa\left( X^{(p^n)} \right)=k \otimes_k \kappa(X)$
(where $k$ is seen as a $k$-algebra via the Frobenius morphism $\phi^n_k$). By 
definition of $C$ we 
have $\kappa(X)=\kappa(C)$, take $n=n'(X)$; then 
$\kappa\left(X^{(p^n)}\right)=\kappa(\mathbb{P}^1_K)$. Thus,
${\kappa\left(\mathbb{P}^1_K\right) = k\otimes_k \kappa(C)}$.
With this identification, the image of
$\phi^n_{\kappa(\mathbb{P}^1_K)}:x \in \kappa\left(\mathbb{P}^1_K\right) \mapsto x^{p^n} 
\in \kappa\left(\mathbb{P}^1_K\right)$ is contained in $\kappa(C)$.

The discrete valuation ring $\mathcal{O}_{\mathbb{P}^1_{K},\infty}$ is defined by the 
valuation ${\rm mult}_{\infty}$ on $\kappa\left(\mathbb{P}^1_K\right)$, and
${\rm mult}_{\infty}$ is an extension of the valuation ${\rm mult}_{P_\infty}$ on
$\kappa(C)$. If $x \in \mathcal{O}_{\mathbb{P}^1_{K},\infty}$, then of course
$x^{p^n} \in \mathcal{O}_{\mathbb{P}^1_{K},\infty}$; and we have shown 
that $x^{p^n} \in  \kappa(C)$, so
$x^{p^n} \in \mathcal{O}_{C, P_\infty}\subset \mathcal{O}_{C_K, P_\infty}$.

So according to the equation \eqref{eq_mu} $\mu^{Z^K/Y^K}(\overline{k})$ is of
$p^{n'(X)}$-torsion.
\end{proof}

To conclude we have shown the following result:
\begin{Pro}\label{Pro_Pic0}
The algebraic group ${\rm Pic}^0_{C/k}$ is unipotent of
$p^{n'(X)}$-torsion, and ${\rm Pic}^0_{C_{k'}/k'}$ is $k'$-split.
\end{Pro}

\subsection{Application to the Picard functor of the regular completion}
\label{sec_pic_C}

\begin{The} \label{th_pic_X}
Let $X$ be a form of $\mathbb{A}^1_k$ and $C$ be the regular completion of
$X$.

Then ${\rm Pic}^0_{C/k}$ is a smooth connected unipotent algebraic group of 
$p^{n'(X)}$-torsion which is $k$-wound and splits over $k'$ (the smallest 
field such that $X_{k'} \cong \mathbb{A}^1_{k'}$).

Moreover if $X$ is a principal homogeneous space for a form $G$ of
$\mathbb{G}_{a,k}$, then
\[
{\rm dim} \ { \rm Pic}^0_{C/k} \leq 
\frac{(p^{{\rm min}(n,m)}-1)(p^{{\rm max}(n,m)}-2)}{2}
\]
where $n=n(G)$ and $m$ is the smallest integer such that $G$ is defined by an 
equation of the form $y^{p^n}=x+a_1x^p+\dots+a_mx^{p^m}$.

In addition if $X$ is a non trivial form of $\mathbb{G}_{a,k}$ and $p\neq 2$, 
then $k'$ is the minimal field extension of $k$ such that ${\rm Pic}^0_{C/k}$ 
splits over $k'$.
\end{The}

\begin{proof}
The assertion on the torsion and the fact that ${\rm Pic}^0_{C/k}$ is 
unipotent and splits over $k'$ are direct consequences of Proposition 
\ref{Pro_Pic0}. According to \cite[Pro.~8.4.2]{NM} ${\rm Pic}^0_{C/k}$ is 
smooth and by \cite[Th.~8.4.1]{NM}, 
${\rm dim} \ { \rm Pic}_{C/k} =
{\rm dim}_k \ {\rm H}^1(C,\mathcal{O}_C)=p_a(C)$.
The variety $C$ is normal and geometrically integral, so according
to \cite[Pro.~9.2.4]{NM} and \cite[Pro.~B.3.2]{CGP} the unipotent algebraic 
group ${\rm Pic}^0_{C/k}$ is $k$-wound.

In the case where $X$ is a principal homogeneous space for a form $G$ of
$\mathbb{G}_{a,k}$, the assertion on the dimension of ${ \rm Pic}^0_{C/k}$ is 
a direct consequence of Corollary \ref{cor_deg_genre}, in view of the fact 
that $C_{k_s}$ is still regular \cite[Cor.~6.14.2]{EGAIV_2} and that the 
arithmetic genus is invariant by field extensions.

We will now show the last assertion. Let $K$ be a field such that
$k \subset K \varsubsetneq k'$, we will show that the unipotent group
${\rm Pic}_{C/k}^0$ does not split on $K$, or equivalently that 
${\rm Pic}_{C_K/K}^0$ is not split. First of all if $C_K$ is normal, then the 
unipotent group ${\rm Pic}_{C_K/K}^0$ is wound, so in particular it is not 
split. Else let ${g:C^{K}\rightarrow C_{K}}$ be the normalisation of $C_{K}$. 
We are going to make the same conductor base construction as in Subsection 
\ref{tor_pic_funct}. Let $\mathcal{C}$ 
be the conductor of $\mathcal{O}_{C_{K}}$ in $\mathcal{O}_{C^{K}}$ i.e. the 
sheaf defined by:
\[\mathcal{C}(U)=\left\lbrace
a \in \mathcal{O}_{C^{K}}\left(g^{-1}(U)\right) \ | \ 
a.\mathcal{O}_{C^{K}}\left(g^{-1}(U)\right) \subset \mathcal{O}_{C_{K}}(U)
\right\rbrace.\]
Then $\mathcal{C}$ is a sheaf of ideals of $\mathcal{O}_{C_{K}}$, and of
$\mathcal{O}_{C^{K}}$. Let $Y$ be the closed sub scheme of $C_K$ defined by 
the sheaf $\mathcal{C}$, let $Z$ be the fibre product
$Y \times_{C_{K}} C^{K}$. Then we have a cocartesian square of $K$-varieties:
\[
\xymatrix{
Z \ar[r] \ar[d]& C^{K} \ar[d]^g \\
Y \ar[r] & C_{K}
}\]
which satisfies the hypotheses of Theorem \ref{th_pic_final}. So we have an 
exact sequence of algebraic groups over $K$
\[
0 \rightarrow \mu^{Z/Y} \rightarrow {\rm Pic}_{C_{K}/K}^0
\rightarrow {\rm Pic}_{C^{K}/K}^0\rightarrow 0.
\]
By hypothesis $K\varsubsetneq k'$ and $p>2$, so $C^{K}$ is not isomorphic to
$\mathbb{P}^1_{K}$ (else according to Lemma \ref{X_P1=>p=2}, we would have
$X_K\cong \mathbb{G}_{a,K}$). Thus ${\rm Pic}_{C^{K}/K}^0$ is a non trivial
$K$-wound algebraic group. Every morphism from a connected smooth unipotent 
split algebraic group to a connected smooth unipotent wound algebraic group is 
trivial \cite[B.3.4]{CGP}, thus ${\rm Pic}_{C_{K}/K}^0$ is not $K$-split.
\end{proof}

\subsection{Rigidified Picard functor}

Let $X$ be a form of $\mathbb{A}^1_k$, let $C$ be the regular completion of
$X$ and let $P_\infty$ be the unique point of $X\setminus C$ (Lemma 
\ref{rus1.1}).

A geometric invariant of $X$ is the rigidified Picard functor
$({\rm Pic}_{C/k}, Y)$ where $Y \subset C$ is a rigidificator of
${\rm Pic}_{C/k}$. In fact the rigidified 
Picard functor has the remarkable property of being "invariant by cocartesian square", i.e. if
\[
\xymatrix{
Y' \ar[r]^v \ar[d]_g & X' \ar[d]^f \\
Y \ar[r]^u & C
}
\]
is a commutative diagram of rigidificators, cocartesian in the category of ringed spaces,
then according to \cite[Lem.~2.2]{MB_pic_var},
$f^*:({\rm Pic}_{C/k}, Y) \rightarrow ({\rm Pic}_{X'/k}, Y')$ is an isomorphism.

According to Proposition \ref{pro_loc_alg} the sequence
\[
0\rightarrow  \mu^{Y}/\mu^{C} \rightarrow ({\rm Pic}_{C/k}, Y)^0 \rightarrow 
{\rm Pic}_{C/k}^0 \rightarrow 0
\]
is an exact sequence of algebraic groups.

\begin{Pro}
If $Y={\rm Spec}(\kappa(P_\infty))$, then $Y$ is a rigidificator of $C$ and
$({\rm Pic}_{C/k}, Y)^0$ is a unipotent $k$-wound algebraic group which splits over $k'$. 
\end{Pro}

\begin{proof}
The algebraic group $\mu^{C}$ is isomorphic to $\mathbb{G}_m$, so
$\mu^{Y}/\mu^{C}\cong \mu^{\kappa(P_\infty)/k}$ is a unipotent algebraic group 
which is $k$-wound according to Remark \ref{rem_ploye}, and which splits over
$\kappa(P_\infty)\subset k'$. The group ${\rm Pic}_{C/k}^0$ is $k$-wound 
unipotent and splits over $k'$ according to Theorem \ref{th_pic_X}.
So the algebraic group $({\rm Pic}_{C/k}, Y)^0$ is an extension of two $k$-
wound algebraic groups, so $({\rm Pic}_{C/k}, Y)^0$ is a $k$-wound unipotent 
group  \cite[V.3.5]{Oe}. Moreover $\mu^{\kappa(P_\infty)/k}$ and
${\rm Pic}_{C/k}^0$ split over $k'$, so $({\rm Pic}_{C/k}, Y)^0$ also splits 
over $k'$.
\end{proof}

{\bf Acknowledgements}: I would like to thank Bruno Laurent and Lara Thomas
for helping me to conjecture Corollary \ref{cor_deg_genre}, as well as
Jean Fasel, Philippe Gille and Burt Totaro for their suggestions. 
I would like to thank Masayoshi Miyanishi for informing me of the existence
of the important reference \cite{FALF} (after a first version of this paper 
was circulated), and for very helpful discussions.
Moreover, I would very much like to thank Michel Brion for extremely useful 
ideas and remarks during the preparation of this article.

\end{document}